\documentclass[11pt]{article}
\usepackage[english]{babel}
\usepackage{amssymb,amsmath,amsthm}
\newtheorem{theorem}{Theorem}[section]
\newtheorem{lemma}[theorem]{Lemma}

\newtheorem{question}[theorem]{Problem}

{\theoremstyle{definition}\newtheorem{definition}[theorem]{Definition}}
{\theoremstyle{definition}}

\numberwithin{equation}{section}

\newtheorem*{thmg}{Theorem G}

\def\C{{\mathbb C}}
\def\D{{\mathbb D}}
\def\N{{\mathbb N}}
\def\Z{{\mathbb Z}}
\def\R{{\mathbb R}}
\def\K{{\mathbb K}}

\def\pd{p_{{}_{\scriptstyle D}}}
\def\epsilon{\varepsilon}
\def\phi{\varphi}
\def\leq{\leqslant}
\def\geq{\geqslant}
\def\ker{{\tt ker}\,}
\def\spann{\hbox{\tt span}\,}

\def\det{{\tt det}\,}

\title{Hypercyclic operators on countably dimensional spaces}

\author{Andre Schenke and Stanislav Shkarin}

\date{}

\begin{document}

\maketitle

\begin{abstract}According to Grivaux, the group $GL(X)$ of
invertible linear operators on a separable infinite dimensional
Banach space $X$ acts transitively on the set $\Sigma(X)$ of
countable dense linearly independent subsets of $X$. As a
consequence, each $A\in \Sigma(X)$ is an orbit of a hypercyclic
operator on $X$. Furthermore, every countably dimensional normed
space supports a hypercyclic operator.

We show that for a separable infinite dimensional Fr\'echet space
$X$, $GL(X)$ acts transitively on $\Sigma(X)$ if and only if $X$
possesses a continuous norm. We also prove that every countably
dimensional metrizable locally convex space supports a hypercyclic
operator.
\end{abstract}

\normalsize

\section{Introduction \label{s1}}\rm

All vector spaces in this article are over the field $\K$, being
either the field $\C$ of complex numbers or the field $\R$ of real
numbers. As usual, $\N$ is the set of positive integers and
$\Z_+=\N\cup\{0\}$. Throughout the article, all topological spaces
{\it are assumed to be Hausdorff}. For a topological vector space
$X$, $L(X)$ is the algebra of continuous linear operators on $X$,
$X'$ is the space of continuous linear functionals on $X$ and
$GL(X)$ is the group of $T\in L(X)$ such that $T$ is invertible and
$T^{-1}\in L(X)$. By saying 'countable', we always mean 'infinite
countable'. Recall that a {\it Fr\'echet space} is a complete
metrizable locally convex space. Recall also that the topology
$\tau$ of a topological vector space $X$ is called {\it weak} if
$\tau$ is exactly the weakest topology making each $f\in Y$
continuous for some linear space $Y$ of linear functionals on $X$
separating points of $X$. It is well-known and easy to see that a
topology of a metrizable infinite dimensional topological vector
space $X$ is weak if and only if $X$ is isomorphic to a dense linear
subspace of $\omega=\K^\N$.

Recall that $x\in X$ is called a {\it hypercyclic vector} for $T\in
L(X)$ if the orbit
$$
O(T,x)=\{T^nx:n\in\Z_+\}
$$
is dense in $X$ and $T$ is called {\it hypercyclic} if it has a
hypercyclic vector. It is easy to see that an orbit of a hypercyclic
vector is always dense countable and linearly independent. For a
topological vector space $X$, we denote the set of all countable
dense linearly independent subsets of $X$ by the symbol $\Sigma(X)$.
Thus
\begin{equation*}
O(T,x)\in \Sigma(X)\ \ \text{if $x$ is a hypercyclic vector for
$T$.}
\end{equation*}
For more information on hypercyclicity see books
\cite{bama-book,book2} and references therein. For the sake of
brevity, we shall say that a subset $A$ of a topological vector
space $X$ is an {\it orbit} if there are $T\in L(X)$ and $x\in X$
such that $A=O(T,x)$.

The starting point for this article is the theorem by Grivaux
\cite{gri2} stating that every countable dense linearly independent
subset of a separable infinite dimensional Banach space is an orbit
of a hypercyclic operator. This result easily follows from another
theorem in \cite{gri2}:

\begin{thmg} For every separable infinite dimensional Banach space
$X$, $GL(X)$ acts transitively on $\Sigma(X)$.
\end{thmg}

The above theorem leads to the following definition.

\begin{definition}\label{gsp} A locally convex topological vector
space $X$ is called a {\it G-space} if $\Sigma(X)$ is non-empty and
$GL(X)$ acts transitively on $\Sigma(X)$.
\end{definition}

Thus Theorem~G states that every separable infinite dimensional
Banach space is a G-space. For the convenience of the reader we
reproduce the derivation of the main result in \cite{gri2} from
Theorem~G.

\begin{lemma}\label{nle} Let $X$ be a G-space possessing a
hypercyclic operator $T_0\in L(X)$. Then every $A\in\Sigma(X)$ is an
orbit.
\end{lemma}

\begin{proof} Let $x_0$ be a hypercyclic vector for $T_0\in L(X)$
and $A\in\Sigma(X)$. Since $X$ is a $G$-space and both $A$ and
$O(T_0,x)$ belong to $\Sigma(X)$, there is $J\in GL(X)$ such that
$J^{-1}(A)=O(T_0,x_0)$. Let $T=JT_0J^{-1}$ and $x=J x_0$. Then
$$
O(T,x)=O(JT_0J^{-1},Jx_0)=\{JT_0^nx:n\in\Z_+\}=J(O(T_0,x_0))=A.
$$
Thus $A$ is an orbit.
\end{proof}

Our main goal is to try and extend Theorem~G to Fr\'echet spaces.
Since every separable infinite dimensional Fr\'echet space supports
a hypercyclic operator \cite{bope}, Lemma~\ref{nle} implies that for
every Fr\'echet space $X$, which is also a G-space, every
$A\in\Sigma(X)$ is an orbit of a hypercyclic operator. Note that
Bonet, Frerick, Peris and Wengenroth \cite{fre} constructed $A\in
\Sigma(\omega)$, which is not an orbit of a hypercyclic operator and
therefore $\omega$ is not a G-space. Our main result is the
following theorem.

\begin{theorem}\label{mama} Let $X$ be a separable infinite dimensional
Fr\'echet space. Then the following statements are equivalent$:$
\begin{itemize}\itemsep=-2pt
\item[\rm(\ref{mama}.1)]$X$ possesses a continuous norm$;$
\item[\rm(\ref{mama}.2)]$X$ is a G-space$;$
\item[\rm(\ref{mama}.3)]every $A\in\Sigma(X)$ is an orbit.
\end{itemize}
\end{theorem}

Our methods can be applied to obtain the following result, which
answers a question raised in \cite{shh} and extends the above
mentioned result of Grivaux stating that every countably dimensional
normed space possesses a hypercyclic operator.

\begin{theorem}\label{mamama} Every countably dimensional metrizable
locally convex space possesses a hypercyclic operator.
\end{theorem}

The following result is our main instrument. In order to formulate
it we need to recall a few definitions. A subset $D$ of a locally
convex space $X$ is called a {\it disk} if $D$ is bounded, convex
and balanced (=is stable under multiplication by $\lambda\in\K$ with
$|\lambda|\leq1$). The symbol $X_D$ stands for the space $\spann(D)$
endowed with the norm $\pd$ being the Minkowski functional of the
set $D$. Boundedness of $D$ implies that the topology of $X_D$ is
stronger than the one inherited from $X$. A disk $D$ in $X$ is
called a {\it Banach disk} if the normed space $X_D$ is complete. It
is well-known that a sequentially complete disk is a Banach disk,
see, for instance, \cite{bonet}. In particular, every compact or
sequentially compact disk is a Banach disk.

We say that a seminorm $p$ on a vector space $X$ is {\it
non-trivial} if $X/\ker p$ is infinite dimensional, where
$$
\ker p=\{x\in X:p(x)=0\}.
$$
Note that the topology of a locally convex space is non-weak if and
only if there is a non-trivial continuous seminorm on $X$.

If $p$ is a seminorm on a vector space $X$, we say that $A\subset X$
is $p$-{\it independent} if $p(z_1a_1+{\dots}+z_na_n)\neq 0$ for any
$n\in\N$, any pairwise different $a_1,\dots,a_n\in A$ and any
non-zero $z_1,\dots,z_n\in\K$. In other words, vectors $x+\ker p$
for $x\in A$ are linearly independent in $X/\ker p$.

\begin{theorem}\label{male}
Let $X$ be a locally convex space, $p$ be a continuous seminorm on
$X$, $D$ be a Banach disk in $X$ and $A,B$ be countable subsets of
$X$ such that both $A$ and $B$ are $p$-independent and both $A$ and
$B$ are dense subsets of the Banach space $X_D$. Then there exists
$J\in GL(X)$ such that $J(A)=B$ and $Jx=x$ for every $x\in\ker p$.
\end{theorem}

We prove Theorem~\ref{male} in Section~\ref{s2}. In Section~\ref{s3}
we show that $GL(\omega)$ acts transitively on the set of dense
countably dimensional subspaces of $\omega$. Section~\ref{s4} is
devoted to the proof of Theorem~\ref{mama}. We prove
Theorem~\ref{mamama} in Section~\ref{s5} and discuss open problems
in Section~\ref{s6}.

\section{Proof of Theorem~\ref{male}\label{s2}}

For a continuous seminorm $p$ on a locally convex space $X$, we
denote
$$
X'_p=\{f\in X':p^*(f)=\sup\{|f(x)|:x\in X,\ p(x)\leq 1\}<\infty\}.
$$
Note that $p^*(f)$ (if finite) is the smallest non-negative number
$c$ such that $f(x)\leq cp(x)$ for every $x\in X$.

\begin{lemma}\label{nuc1} Let $X$ be a locally convex space, $p$ be a
continuous seminorm on $X$, $D$ be a Banach disk in $X$ and
$\{x_n\}_{n\in\N}$ and $\{f_n\}_{n\in\N}$ be sequences in $X_D$ and
$X'$ respectively such that $c=\sum\limits_{n=1}^\infty
p^*(f_n)p_D(x_n)<\infty$. Then the formula
$Tx=\sum\limits_{n=1}^\infty f_n(x)x_n$ defines a continuous linear
operator on $X$. Furthermore, if $p$ is bounded by $1$ on $D$ and
$c<1$, then the operator $I+T$ is invertible.
\end{lemma}

\begin{proof} Clearly $p_D(f_n(x)x_n)=|f_n(x)|p_D(x_n)\leq
p(x)p^*(f_n)p_D(x_n)$. It follows that the series defining $Tx$
converges in $X_D$ and
$$
p_D(Tx)\leq cp(x)\ \ \text{for every $x\in X$.}
$$
Thus $T$ is a well-defined continuous linear map from $X$ to $X_D$.
Since the topology of $X_D$ is stronger than the one inherited from
$X$, $T\in L(X)$.

Assume now that $p$ is bounded by $1$ on $D$ and $c<1$. Then from
the inequality $p(x)\leq p_D(x)$ and the above display it follows
that $p_D(T^nx)\leq c^np(x)$ for every $x\in X$ and $n\in\N$. Since
$c<1$, the formula $Sx=\sum\limits_{n=1}^\infty (-T)^nx$ defines a
linear map from $X$ to $X_D$ satisfying $p_D(Sx)\leq
\frac{c}{1-c}p(x)$ for $x\in X$. Thus $S$ is continuous as a map
from $X$ to $X_D$ and therefore $S\in L(X)$. It is a routine
exercise to check that $(I+S)(I+T)=(I+T)(I+S)=I$. That is, $I+T$ is
invertible.
\end{proof}

\begin{lemma}\label{step1}Let $\epsilon>0$, $X$ be a locally convex
space, $D$ be a Banach disk in $X$, $Y$ be a closed linear subspace
of $X$, $M\subseteq Y\cap X_D$ be a dense subset of $Y$ such that
$M$ is $p_D$-dense in $Y\cap X_D$, $p$ be a continuous seminorm on
$X$, $L$ be a finite dimensional subspace of $X$ and $T\in L(X)$ be
a finite rank operator such that $T(Y)\subseteq Y\cap X_D$, $T(\ker
p)\subseteq\ker p$ and $\ker (I+T)=\{0\}$. Then
\begin{itemize}
\item[{\rm (1)}]for every $u\in Y\cap X_D$ such that
$(u+L)\cap \ker p=\varnothing$, there are $f\in X'$ and $v\in Y\cap
X_D$ such that $p^*(f)=1$, $f\bigr|_L=0$, $p_D(v)<\epsilon$,
$(I+R)u\in M$ and $\ker(I+R)=\{0\}$, where $Rx=Tx+f(x)v;$
\item[{\rm (2)}]for every $u\in Y\cap X_D$ such that
$(u+(I+T)(L))\cap\ker p=\varnothing$, there are $f\in X'$, $a\in M$
and $v\in Y\cap X_D$ such that $p^*(f)=1$, $f\bigr|_L=0$,
$p_D(v)<\epsilon$, $(I+R)a=u$ and $\ker(I+R)=\{0\}$, where
$Rx=Tx+f(x)v.$
\end{itemize}
\end{lemma}

\begin{proof} Let $u\in Y\cap X_D$ be such that
$(u+L)\cap \ker p=\varnothing$. The Hahn--Banach theorem provides
$f\in X'$ such that $p^*(f)=1$, $f(u)\neq 0$ and $f\bigr|_L=0$.
First, observe that there is $\delta>0$ such that the only solution
of the equation $Tx+x+f(x)v=0$ is $x=0$ whenever $v\in X_D$ and
$p_D(v)<\delta$. Indeed, assume the contrary. Then there exist
sequences $\{x_n\}_{n\in\N}$ and $\{v_n\}_{n\in\N}$ in $X_D$ such
that $p_D(x_n)=1$ for every $n\in\N$, $p_D(v_n)\to 0$ and
$Tx_n+x_n+f(x_n)v_n=0$ for each $n\in\N$. Since $\{Tx_n\}_{n\in\N}$
is a bounded sequence in the finite dimensional subspace $T(X)$ of
$X_D$, passing to a subsequence, if necessary, we can without loss
of generality assume that $x_n$ converges to $x\in T(X)\subset X_D$
with respect to $p_D$. That is, $p_D(x)=1$ and $p_D(x_n-x)\to 0$.
Passing to the $p_D$-limit in $Tx_n+x_n+f(x_n)v_n=0$, we obtain
$Tx+x=0$, which contradicts the equality $\ker(I+T)=\{0\}$. Thus
there is $\delta>0$ such that the only solution of the equation
$Tx+x=f(x)v$ is $x=0$ whenever $v\in X_D$ and $p_D(v)<\delta$. Since
$M$ is $p_D$-dense in $Y\cap X_D$ and $u+Tu\in Y\cap X_D$, there is
$r\in M$ such that
$p_D(r-u-Tu)<\min\{\delta|f(u)|,\epsilon|f(u)|\}$. Define
$v=\frac1{f(u)}(r-u-Tu)\in Y\cap X_D$ and $Rx=Tx+f(x)v$. Clearly,
$p_D(v)<\epsilon$. Since $p_D(v)<\delta$, $\ker (I+R)=\{0\}$. A
direct computation gives $(I+R)u=u+Tu+\frac1{f(u)}f(u)(r-u-Tu)=r\in
M$. Thus $f$ and $v$ satisfy all desired conditions.

Now assume that $u\in Y\cap X_D$ is such that $(u+(I+T)(L))\cap\ker
p=\varnothing$. The fact that $T$ has finite rank and $\ker
(I+T)=\{0\}$ implies that $I+T$ is invertible. Furthermore, the
inclusion $T(X)\subseteq Y\cap X_D$ and the finiteness of the rank
of $T$ imply that $(I+T)^{-1}(Y)\subseteq Y$ and
$(I+T)^{-1}(X_D)\subseteq X_D$. Thus there is a unique $w\in Y\cap
X_D$ such that $(I+T)w=u$. Since $(u+(I+T)(L))\cap\ker
p=\varnothing$ and $T(\ker p)\subseteq\ker p$, we have  $(w+L)\cap
\ker p=\varnothing$. The Hahn--Banach theorem provided $f\in X'$
such that $p^*(f)=1$, $f(w)\neq 0$ and $f\bigr|_L=0$. Exactly as in
the first part of the proof, we observe that there is $\delta>0$
such that the only solution of the equation $Tx+x+f(x)v=0$ is $x=0$
whenever $v\in X_D$ and $p_D(v)<\delta$. Since $f(w)\neq 0$ and $M$
is $p_D$-dense in $Y\cap X_D$, we can find $a\in M$ (close enough to
$w$ with respect to $p_D$) such that $p_D(v)<\delta$ and
$p_D(v)<\delta$, where $v=\frac1{f(a)}(I+T)(w-a)$. Now set
$Rx=Tx+f(x)v$. Since $p_D(v)<\delta$, $\ker (I+R)=\{0\}$. A direct
computation gives $(I+R)a=a+Ta+\frac1{f(a)}f(a)(I+T)(w-a)=(I+T)w=u$.
Thus $f$, $a$ and $v$ satisfy all desired conditions.
\end{proof}

Now we are ready to prove Theorem~\ref{male}. Without loss of
generality, we may assume that $p$ is bounded by $1$ on $D$.
Equivalently, $p(x)\leq p_D(x)$ for every $x\in X_D$.

Fix arbitrary bijections $a:\N\to A$ and $b:\N\to B$ and a sequence
$\{\epsilon_n\}_{n\in\N}$ of positive numbers such that
$\sum\limits_{n=1}^\infty\epsilon_n<1$. We shall construct
inductively sequences $\{n_k\}_{k\in\N}$ and $\{m_k\}_{k\in\N}$ of
positive integers, $\{v_k\}_{k\in\N}$ in $X_D$, $\{f_k\}_{k\in\N}$
in $X_p'$ and $\{T_n\}_{n\in\Z_+}$ in $L(X)$ such that $T_0=0$ and
for every $k\in\N$,
\begin{align}
&m_j\neq m_l\ \ \text{and}\ \ n_j\neq n_l\ \ \text{for $1\leq
j<l\leq 2k$}; \label{t1}
\\
&\{1,\dots,k\}\subseteq \{n_1,\dots,n_{2k}\}\cap
\{m_1,\dots,m_{2k}\}\ \ \text{for $k\geq 1$}; \label{t2}
\\
&\text{$p_D(v_j)<\epsilon_j$ and $p^*(f_j)\leq 1$ for $1\leq j\leq
2k$;} \label{t3}
\\
&(T_k-T_{k-1})x=f_{2k-1}(x)v_{2k-1}+f_{2k}(x)v_{2k}; \label{t4}
\\
&(I+T_k)a(n_j)=b(m_j)\ \ \text{for $1\leq j\leq 2k$}. \label{t5}
\end{align}

$T_0=0$ serves as the basis of induction. Let $q\geq 1$ and assume
that $m_j,n_j,v_j,f_j,T_j$ for $j\leq 2q-2$ satisfying
(\ref{t1}--\ref{t5}) are already constructed. By (\ref{t4}) with
$k<q$,
$$
T_{q-1}x=\sum_{j=1}^{2q-2} f_j(x)v_j.
$$
According to (\ref{t3}) with $k=q-1$,
$\sum\limits_{j=1}^{2q-2}p^*(f_j)p_D(v_j)<\sum\limits_{j=1}^{2n-2}\epsilon_j<1$.
By Lemma~\ref{nuc1}, $I+T_{q-1}$ is invertible. Since each
$p^*(f_j)$ is finite, $T_{q-1}$ vanishes on $\ker p$. In particular,
$T_{q-1}(\ker p)\subseteq \ker p$. Since each $v_j$ belongs to
$X_D$, $T_{q-1}(X)\subseteq X_D$. Clearly $T_{q-1}$ has finite rank.
Set $Y=X$, $n_{2q-1}=\min (\N\setminus \{n_1,\dots,n_{2q-2}\})$ and
$m_{2q}=\min(\N\setminus \{m_1,\dots,m_{2q-2}\})$. Since $A$ is
$p$-independent $(u+L)\cap \ker
 p=\varnothing$, where $u=a(n_{2q-1})$ and $L=\spann\{a(n_1),\dots,a_{n_{2q-2}}\}$.
Applying the first part of Lemma~\ref{step1} with the just defined
$u$, $Y$ and $L$ and with $T=T_{q-1}$, $\epsilon=\epsilon_{2q-1}$
and $M=B\setminus\{b(m_{2q}),b(m_1),b(m_2),\dots,b(m_{2q-2})\}$, we
find $f_{2q-1}\in X'$ and $v_{2q-1}\in X_D$ such that
$p^*(f_{2q-1})=1$, $f_{2q-1}\bigr|_L=0$,
$p_D(v_{2q-1})<\epsilon_{2q-1}$ and $(I+S)u\in M$, where
$Sx=T_{q-1}x+f_{2q-1}(x)v_{2q-1}$. The inclusion $(I+S)u\in M$ means
that $(I+S)u=b(m_{2q-1})$ for some
$m_{2q-1}\in\N\setminus\{m_{2q},m_1,m_2,\dots,m_{2q-2}\}$. Since
$u=a(n_{2q-1})$ and $f_{2q-1}\bigr|_L=0$, (\ref{t5}) with $k=q-1$
implies that
$$
(I+S)a(m_j)=b(n_j)\ \ \text{for}\ \ 1\leq j\leq 2q-1.
$$
By definition of $S$, $Sx=\sum\limits_{j=1}^{2q-1} f_j(x)v_j$ with
$\sum\limits_{j=1}^{2q-1}p^*(f_j)p_D(v_j)<\sum\limits_{j=1}^{2q-1}\epsilon_j<1$.
By Lemma~\ref{nuc1}, $I+S$ is invertible. Since each $p^*(f_j)$ is
finite and each $v_j$ belongs to $X_D$, $S$ vanishes on $\ker p$ and
$S(X)\subseteq X_D$. Clearly $S$ has finite rank. Since $B$ is
$p$-independent the above display ensures that $(u+(I+S)(L))\cap
\ker p=\varnothing$, where $u=b(m_{2q})$ and $L=\spann\{a(n_j):1\leq
j\leq 2q-1\}$. Applying the second part of Lemma~\ref{step1} with
$Y=X$ and the just defined $u$ and $L$ and with $T=S$,
$\epsilon=\epsilon_{2q}$ and $M=A\setminus\{a(n_j):1\leq j\leq
2q-1\}$, we find $f_{2q}\in X_p'$, $v_{2q}\in X_D$ and $w\in M$ such
that $p^*(f_{2q})=1$, $f_{2q}\bigr|_L=0$,
$p_D(v_{2q})<\epsilon_{2q}$ and $(I+T_{q})w=u$, where
$T_qx=Sx+f_{2n}(x)v_{2n}$. The inclusion $w\in M$ means that
$w=a(n_{2q})$ for some  $n_{2q}\in
\N\setminus\{n_1,\dots,n_{2q-1}\}$. Since $u=b(m_{2q})$ and
$f_{2n}\bigr|_L=0$, the above display yields
$$
(I+T_q)a(m_j)=b(n_j)\ \ \text{for}\ \ 1\leq j\leq 2q.
$$
Since $n_{2q-1}\neq n_{2q}$, $m_{2q-1}\neq m_{2q}$,
$n_{2q-1},n_{2q}\notin\{n_1,\dots,n_{2q-2}\}$ and
$m_{2q-1},m_{2q}\notin\{n_1,\dots,m_{2q-2}\}$, (\ref{t1}) with $k=q$
follow from (\ref{t1}) with $k=q-1$. By construction, (\ref{t3}),
(\ref{t4}) and (\ref{t5}) with $k=q$ are satisfied. Since
$n_{2q-1}=\min (\N\setminus \{n_1,\dots,n_{2q-2}\})$ and
$m_{2q}=\min(\N\setminus \{m_1,\dots,m_{2q-2}\})$, (\ref{t2}) for
$k=q$ follows from (\ref{t2}) with $k=q-1$. Thus
(\ref{t1}--\ref{t5}) are all satisfied for $k=q$. This concludes the
inductive construction of $m_j,n_j,v_j,f_j,T_j$ for $j\in\N$.

By (\ref{t1}) and (\ref{t2}), the map $n_j\mapsto m_j$ is a
bijection from $\N$ to itself. By (\ref{t3}) and (\ref{t4}), the
sequence $\{T_n\}$ converges pointwise to the operator $T\in L(X)$
given by the formula $Tx=\sum\limits_{j=1}^\infty f_j(x)v_j$. Since
$\sum\epsilon_j<1$, (\ref{t3}) and Lemma~\ref{nuc1} imply that
$J=I+T$ is invertible. Since $p^*(f_j)<\infty$ for every $j$, $T$
vanishes on $\ker p$ and therefore $Jx=x$ for $x\in\ker p$. Passing
to the limit in (\ref{t5}), we obtain that $Ja(n_j)=b(m_j)$ for
every $j\in\N$. Since $n_j\mapsto m_j$ is a bijection from $\N$ to
itself, we get $J(A)=B$. Thus $J$ satisfies all required conditions.
The proof of Theorem~\ref{male} is now complete.

\section{Countably dimensional subspaces of $\omega$ \label{s3}}

The main result of this section is the following theorem.

\begin{theorem}\label{omeg} $GL(\omega)$ acts transitively on the
set of dense countably dimensional linear subspaces of $\omega$.
\end{theorem}

In order to prove the above result we need few technical lemmas. As
usual, we identify $\omega$ with $\K^\N$. For $n\in\N$, the symbol
$\delta_n$ stands for the $n^{\rm th}$ coordinate functional on
$\omega$. That is, $\delta_n\in\omega'$ is defined by
$\delta_n(x)=x_n$. By $\phi$ we denote the linear subspace of
$\omega$ consisting of sequences with finite support. That is,
$x\in\phi$ precisely when there is $n\in\N$ such that $x_m=0$ for
all $m\geq n$.

\begin{lemma}\label{ome1} Let $f_1,\dots,f_{n+1}$ be linearly
independent functionals on a vector space $E$, $A\subseteq E$ be
such that $\spann(A)=E$ and $x_1,\dots,x_n\in E$ be such that the
matrix $\{f_j(x_k)\}_{1\leq j,k\leq n}$ is invertible. Then there
exists $x_{n+1}\in A$ such that $\{f_j(x_k)\}_{1\leq j,k\leq n+1}$
is invertible.
\end{lemma}

\begin{proof} Since the matrix $B=\{f_j(x_k)\}_{1\leq j,k\leq n}$ is invertible,
the vector $(f_{n+1}(x_1),\dots,f_{n+1}(x_n))\in\K^n$ is a linear
combination of the rows of $B$. That is, there exist
$c_1,\dots,c_n\in\K$ such that $g(x_j)=0$ for $1\leq j\leq n$, where
$g=f_{n+1}-\sum\limits_{j=1}^n c_jf_j$. Since $f_j$ are linearly
independent, $g\neq 0$. Since $\spann(A)=E$, we can find $x_{n+1}\in
A$ such that $g(x_{n+1})\neq 0$. Consider the $(n+1)\times(n+1)$
matrix $C=\{\gamma_{j,k}\}_{1\leq j,k\leq n+1}$ defined by
$\gamma_{j,k}=f_j(x_k)$ for $1\leq j\leq n$, $1\leq k\leq n+1$ and
$\gamma_{n+1,k}=g(x_k)$ for $1\leq k\leq n+1$. Since
$\{\gamma_{j,k}\}_{1\leq j,k\leq n}=B$ is invertible,
$\gamma_{n+1,k}=g(x_k)=0$ for $1\leq k\leq n$ and
$\gamma_{n+1,n+1}=g(x_{n+1})\neq 0$, $C$ is invertible. Indeed,
$\det C=g(x_{n+1})\det B$. It remains to notice that
$B^+=\{f_j(x_k)\}_{1\leq j,k\leq n+1}$ is obtained from $C$ by
adding a linear combination of the first $n$ rows to the last row.
Hence $\det B^+=\det C\neq 0$ and $B^+$ is invertible as required.
\end{proof}

Applying Lemma~\ref{ome1} and treating the elements of a vector
space $E$ as linear functionals on a space of linear functionals on
$E$, we immediately get the following result.

\begin{lemma}\label{ome2} Let $x_1,\dots,x_{n+1}$ be linearly
independent elements of a vector space $E$, $A$ be a collection of
linear functionals on $E$ separating the points of $E$ and
$f_1,\dots,f_n$ be linear functionals on $E$ such that the matrix
$\{f_j(x_k)\}_{1\leq j,k\leq n}$ is invertible. Then there exists
$f_{n+1}\in A$ such that $\{f_j(x_k)\}_{1\leq j,k\leq n+1}$ is
invertible.
\end{lemma}

\begin{lemma}\label{ome3} Let $\{u_n\}_{n\in\N}$ be a Hamel basis in a
vector space $E$ and $\{f_n\}_{n\in\N}$ be a linearly independent
sequence of linear functionals on $E$ separating the points of $E$.
Then there exist bijections $\alpha:\N\to\N$ and $\beta:\N\to\N$
such that for every $n\in\N$, the matrix
$\{f_{\alpha(j)}(x_{\beta(k)})\}_{1\leq j,k\leq n}$ is invertible.

Furthermore, there exist complex numbers $c_{j,k}$ for $j\leq k$
such that $c_{j,j}\neq 0$ and $f_{\alpha(j)}(v_j)=1$ for $j\in\N$
and $f_{\alpha(j)}(v_k)=0$ for $j,k\in\N$ and $j<k$, where
$v_k=\sum\limits_{m=1}^k c_{m,k}u_{\beta(m)}$.
\end{lemma}

\begin{proof} We shall construct inductively two sequences
$\{\alpha_j\}_{j\in\N}$ and $\{\beta_k\}_{k\in\N}$ of natural
numbers such that
\begin{align}
\{1,\dots,n\}\subseteq\{\alpha_1,\dots,\alpha_{2n}\}\cap\{\beta_1,\dots,\beta_{2n}\}\
\ \text{for each $n\in\N$};\label{inin1}
\\
\text{$\{f_{\alpha_j}(x_{\beta_k})\}_{1\leq j,k\leq n}$ is
invertible for every $n\in\N$.}\label{inin2}
\end{align}

{\bf Basis of induction.} \ Take $\alpha_1=1$. Since $f_1\neq 0$ and
the vectors $u_n$ span $E$, there is $\beta_1\in\N$ such that
$f_{\alpha_1}(u_{\beta_1})\neq 0$. Now we take
$\beta_2=\min(\N\setminus\{\beta_1\})$. By Lemma~\ref{ome2}, there
is $\alpha_2\in\N$ such that $\{f_{\alpha_j}(x_{\beta_k})\}_{1\leq
j,k\leq 2}$ is invertible. Clearly, $1\in\{\alpha_1,\alpha_2\}\cap
\{\beta_1,\beta_2\}$. Thus $\alpha_1,\alpha_2,\beta_1,\beta_2$
satisfy (\ref{inin1}) and (\ref{inin2}).

{\bf The induction step.} \ Assume that $m\in\N$ and
$\alpha_1,\dots,\alpha_{2m},\beta_1,\dots,\beta_{2m}$ satisfying
(\ref{inin1}) with $n\leq m$ and (\ref{inin2}) with $n\leq 2m$ are
already constructed. The latter implies that $\beta_j$ are pairwise
distinct and $\alpha_j$ are pairwise distinct. First, take
$\alpha_{2m+1}=\min(\N\setminus\{\alpha_1,\dots,\alpha_{2m}\})$. By
Lemma~\ref{ome1}, there is $\beta_{2m+1}\in\N$ such that
$\{f_{\alpha_j}(x_{\beta_k})\}_{1\leq j,k\leq 2m+1}$ is invertible.
Automatically, $\beta_{m+1}\notin\{\beta_1,\dots,\beta_{2m}\}$.
Next, we take
$\beta_{2m+2}=\min(\N\setminus\{\beta_1,\dots,\beta_{2m+1}\})$. By
Lemma~\ref{ome2}, there is $\alpha_{2m+2}\in\N$ such that
$\{f_{\alpha_j}(x_{\beta_k})\}_{1\leq j,k\leq 2m+2}$ is invertible.
Since
$\{1,\dots,m\}\subseteq\{\alpha_1,\dots,\alpha_{2m}\}\cap\{\beta_1,\dots,\beta_{2m}\}$,
$\alpha_{2m+1}=\min(\N\setminus\{\alpha_1,\dots,\alpha_{2m}\})$ and
$\beta_{2m+2}=\min(\N\setminus\{\beta_1,\dots,\beta_{2m+1}\})$, we
have
$\{1,\dots,m+1\}\subseteq\{\alpha_1,\dots,\alpha_{2m+2}\}\cap\{\beta_1,\dots,\beta_{2m+2}\}$.
Thus $\alpha_1,\dots,\alpha_{2m+2},\beta_1,\dots,\beta_{2m+2}$
satisfy (\ref{inin1}) with $n\leq m+1$ and (\ref{inin2}) with $n\leq
2m+2$.

This concludes the inductive construction of $\{\alpha_j\}_{j\in\N}$
and $\{\beta_k\}_{k\in\N}$ satisfying (\ref{inin1}) and
(\ref{inin2}). According to (\ref{inin2}), $\alpha_j$ are pairwise
distinct and $\beta_j$ are pairwise distinct. By (\ref{inin1}),
$\{\alpha_j:j\in\N\}=\{\beta_k:k\in\N\}=\N$. Hence the maps
$\alpha,\beta:\N\to\N$ defined by $\alpha(j)=\alpha_j$ and
$\beta(j)=\beta_j$ are bijections. By (\ref{inin2}), the matrix
$\{f_{\alpha(j)}(x_{\beta(k)})\}_{1\leq j,k\leq n}$ is invertible
for every $n\in\N$.

Now let $m\in\N$. Since $A_m=\{f_{\alpha(j)}(x_{\beta(k)})\}_{1\leq
j,k\leq m}$ is invertible, we can find
$c_{1,m},c_{2,m},\dots,c_{m,m}$ such that the linear combination of
the columns of $A_m$ with the coefficients $c_{1,m},\dots,c_{m,m}$
is the vector $(0,\dots,0,1)$. Note that $c_{m,m}$ can not be $0$.
Indeed, otherwise a non-trivial linear combination of the columns of
the invertible matrix $A_{m-1}$ is $0$. The fact that the linear
combination of the columns of $A_m$ with the coefficients
$c_{1,m},\dots,c_{m,m}$ is $(0,\dots,0,1)$ can be rewritten as
$f_{\alpha(m)}(v_m)=1$ and $f_{\alpha(j)}(v_m)=0$ for $j<m$, where
$v_m=\sum\limits_{j=1}^m c_{j,m}u_{\beta(j)}$. Doing this for every
$m\in\N$, we obtain the numbers $\{c_{j,m}\}$ and the vectors $v_m$
satisfying all desired conditions.
\end{proof}

\begin{lemma}\label{wow} Let $E$  be a dense countably dimensional
linear subspace of $\omega$. Then there is a Hamel basis
$\{v_n\}_{n\in\N}$ in $E$ and a bijection $\alpha:\N\to\N$ such that
$\delta_{\alpha(n)}(v_n)=1$ and $\delta_{\alpha(k)}(v_n)=0$ whenever
$n\in\N$ and $k<n$.
\end{lemma}

\begin{proof} Take an arbitrary Hamel basis $\{u_n\}_{n\in\N}$ in
$E$. Applying Lemma~\ref{ome3} with $f_n=\delta_n$, we find
bijections $\alpha,\beta:\N\to\N$ and complex numbers $c_{j,k}$ for
$j\leq k$ such that $c_{j,j}\neq 0$ and $f_{\alpha(j)}(v_j)=1$ for
$j\in\N$ and $f_{\alpha(j)}(v_k)=0$ for $j,k\in\N$ and $j<k$, where
$v_k=\sum\limits_{m=1}^k c_{m,k}u_{\beta(m)}$.

It remains to notice that since $\{u_n\}$ is a Hamel basis in $E$,
$\{v_n\}$ is also a Hamel basis in $E$. Indeed, it is
straightforward to verify that
$u_{\beta(n)}\in\spann\{v_1,\dots,v_n\}\setminus\spann\{v_1,\dots,v_{n-1}\}$
for every $n\in\N$. Thus the Hamel basis $\{u_n\}$ and the bijection
$\alpha$ satisfy all desired conditions.
\end{proof}

\begin{proof}[{\bf Proof of Theorem~\ref{omeg}}] Let $E$ be a dense
countably dimensional subspace of $\omega$. By Lemma~\ref{wow},
there is a Hamel basis $\{v_n\}_{n\in\N}$ in $E$ and a bijection
$\alpha:\N\to\N$ such that $\delta_{\alpha(n)}(v_n)=1$ and
$\delta_{\alpha(k)}(v_n)=0$ whenever $n\in\N$ and $k<n$. Consider
$T:\omega\to\omega$ defined by the formula
$$
Tx=\sum_{n=1}^\infty x_{\alpha(n)}v_n=\sum_{n=1}^\infty
\delta_{\alpha(n)}(x)v_n.
$$
If $\{e_j\}_{j\in\N}$ is the standard basis of $\omega$, then it is
easy to see that the matrix of $T$ with respect to the 'shuffled'
basis $\{e_{\alpha(j)}\}_{j\in\N}$ is lower-triangular with all
entries $1$ on the main diagonal. It follows that $T$ is a
well-defined invertible continuous linear operator on $\omega$. It
remains to observe that $T(\phi)=E$. Hence each dense countably
dimensional subspace of $\omega$ is the image of $\phi$ under an
isomorphism of $\omega$ onto itself. Hence isomorphisms of $\omega$
act transitively on the set of dense countably dimensional linear
subspaces of $\omega$.
\end{proof}

\section{Proof of Theorem~\ref{mama} \label{s4}}

\begin{lemma}\label{l11} Let $\{x_n\}_{n\in\Z_+}$ be a sequence
in a sequentially complete locally convex space $X$ such that
$x_n\to 0$. Then the set
$$
K=\biggl\{\sum_{n=0}^\infty a_nx_n:a\in\ell_1,\ \|a\|_1\leq
1\biggr\}
$$
is a Banach disk. Moreover, $E=\spann\{x_n:n\in \Z_+\}$ is a dense
linear subspace of the Banach space $X_K$.
\end{lemma}

\begin{proof} Let $Q=\{a\in \ell_1:\|a\|_1\leq 1\}$ be endowed with
the coordinatewise convergence topology. Then $Q$ is metrizable and
compact as a closed subspace of the compact metrizable space
$\D^{\Z_+}$, where $\D=\{z\in\K:|z|\leq 1\}$. Obviously, the map
$\Phi:Q\to K$, $\Phi(a)=\sum\limits_{n=0}^\infty a_nx_n$ is onto.
Moreover, $\Phi$ is continuous. Indeed, let $p$ be a continuous
seminorm on $X$, $a\in Q$ and $\epsilon>0$. Since $x_n\to 0$, there
is $m\in\Z_+$ such that $p(x_n)\leq \epsilon$ for $n>m$. Set
$\delta=\frac{\epsilon}{1+p(x_0)+{\dots}+p(x_m)}{\vrule width0pt
height0pt depth6pt}$ and $W=\{b\in Q:|a_j-b_j|<\delta\ \ \text{for}\
\ 0\leq j\leq m\}$. Then $W$ is a neighborhood of $a$ in $Q$ and for
each $b\in W$, we have
$$
p(\Phi(b)-\Phi(a))=p\biggl(\sum_{n=0}^\infty (b_n-a_n)x_n\biggr)\leq
\sum_{n=0}^\infty |b_n-a_n|p(x_n).
$$
Since $p(x_n)<\epsilon$ for $n>m$, $|a_n-b_n|<\delta$ for $n\leq m$
and $\|a\|_1\leq 1$, $\|b\|_1\leq 1$, we obtain
$$
p(\Phi(b)-\Phi(a))\leq \delta\sum_{n=0}^m
p(x_m)+\epsilon\sum_{n=m+1}^\infty |b_n-a_n|\leq
2\epsilon+\delta\sum_{n=0}^m p(x_m).
$$
Using the definition of $\delta$, we see that
$p(\Phi(b)-\Phi(a))\leq 3\epsilon$. Since $a$, $p$ and $\epsilon$
are arbitrary, $\Phi$ is continuous. Thus $K$ is compact and
metrizable as a continuous image of a compact metrizable space.
Obviously, $K$ is convex and balanced. Hence $K$ is a Banach disk
(any compact disk is a Banach disk). It remains to show that $E$ is
dense in $X_K$. Take $u\in X_K$. Then there is $a\in \ell_1$ such
that $u=\sum\limits_{k=0}^\infty a_kx_k$. Clearly,
$u_n=\sum\limits_{k=0}^n a_kx_k\in E$. Then
$p_K(u-u_n)=p_K\Bigl(\sum\limits_{k=n+1}^\infty
a_kx_k\Bigr)\leq\sum\limits_{k=n+1}^\infty |a_k|\to 0$ as
$n\to\infty$. Hence $E$ is dense in $X_K$.
\end{proof}

\begin{lemma}\label{l22} Let $X$ be a Fr\'echet space and
$A$ and $B$ be dense countable subsets of $X$. Then there exists a
Banach disk $D$ in $X$ such that both $A$ and $B$ are dense subsets
of the Banach space $(X_D,p_D)$.
\end{lemma}

\begin{proof} Let $C$ be the set of all linear combinations of the
elements of $A\cup B$ with rational coefficients. Obviously, $C$ is
countable. Pick a map $f:\N\to C$ such that $f^{-1}(x)$ is an
infinite subset of $\N$ for every $x\in C$. Since $A$ and $B$ are
dense in $X$, we can find maps $\alpha:\N\to A$ and $\beta:\N\to B$
such that $4^m(f(m)-\alpha(m))\to 0$ and $4^m(f(m)-\beta(m))\to 0$.
Since $A$ and $B$ are countable, we can write $A=\{x_{m}:m\in\N\}$
and $B=\{y_{m}:m\in\N\}$. Using metrizability of $X$, we can find a
sequence $\{\gamma_{m}\}_{m\in\N}$ of positive numbers such that
$\gamma_{m} x_{m}\to 0$ and $\gamma_{m} y_{m}\to 0$. Enumerating the
countable set
$$
\{2^m(f(m)-\alpha(m)):m\in\N\}\cup \{2^m(f(m)-\beta(m)):m\in\N\}\cup
\{\gamma_{m} x_{m}:m\in\N\}\cup \{\gamma_{m} y_{m}:m\in\N\}
$$
as one (convergent to 0) sequence and applying Lemma~\ref{l11} to
this sequence, we find that there is a Banach disk $D$ in $X$ such
that $X_D$ contains $A$ and $B$, the linear span of $A\cup B$ is
$p_D$-dense in $X_D$ and $f(m)-\alpha(m)\to 0$ and $f(m)-\beta(m)\to
0$ in $X_D$. The $p_D$-density of the linear span of $A\cup B$ in
$X_D$ implies the $p_D$-density of $C$ in $X_D$. Taking into account
that $f^{-1}(x)$ is infinite for every $x\in C$ and that $\alpha$
takes values in $A$, the $p_D$-density of $C$ in $X_D$ and the
relation $\pd(f(m)-\alpha(m))\to 0$ implies that $A$ is $p_D$-dense
in $X_D$. Similarly, $B$ is $p_D$-dense in $X_D$. Thus $D$ satisfies
all required conditions.
\end{proof}

\begin{lemma}\label{lile} Let $X$ be a separable Fr\'echet space and
$p$ be a non-trivial continuous seminorm on $X$. Then for every
dense countable set $A\subset X$, there is $B\subseteq A$ such that
$B$ is $p$-independent and dense in $X$.
\end{lemma}

\begin{proof} Let $\{U_n\}_{n\in\N}$ be a countable basis of the
topology of $X$. We shall construct (inductively) a sequence
$\{x_n\}_{n\in\N}$ of elements of $A$ such that for every $n\in\N$,
\begin{equation}\label{in1}
\text{$x_n\in U_n$\ \ and\ \ $x_1,\dots,x_n$ are $p$-independent.}
\end{equation}

Note that in every topological vector space, the linear span of a
dense subset of a non-empty open set is a dense linear subspace. It
follows that for each $n\in\N$,
\begin{equation}\label{gn}
\text{a proper closed linear subspace of $X$ can not contain $A\cap
U_n$.}
\end{equation}
Hence $A\cap U_1\not\subseteq \ker p$. Thus we can pick $x_1\in
(A\cap U_1)\setminus \ker p$, which will serve as the basis of
induction. Assume now that $m\in\N$ and $x_1,\dots,x_m$ satisfying
(\ref{in1}) for $n\leq m$ are already constructed. Let $L$ be the
linear span of $x_1,\dots,x_m$. Since the sum of a closed subspace
of a topological vector space and a finite dimensional subspace is
always closed and the codimension of $\ker p$ in $X$ is infinite,
$L+\ker p$ is a proper closed linear subspace of $X$. By (\ref{gn}),
we can pick $x_{m+1}\in (A\cap U_{m+1})\setminus (\ker p+L)$. It is
easy to see that $x_1,\dots,x_m,x_{m+1}$ satisfy (\ref{in1}) for
$n\leq m+1$, which concludes the inductive construction of
$\{x_n\}_{n\in\N}$ satisfying (\ref{in1}) for every $n\in\N$. It
remains to observe that $B=\{x_n:n\in\N\}\subseteq A$, $B$ is dense
in $X$ since it meets each $U_n$ and $B$ is $p$-independent.
\end{proof}

\subsection{Proof of the implications
(\ref{mama}.1)$\Longrightarrow$(\ref{mama}.2) and
(\ref{mama}.2)$\Longrightarrow$(\ref{mama}.3)}

Assume that a separable infinite dimensional Fr\'echet space $X$
possesses a continuous norm $p$ and that $A,B\in \Sigma(X)$. By
Lemma~\ref{l22}, there is a Banach disk $D$ in $X$ such that both
$A$ and $B$ are dense subsets of the Banach space $X_D$. By
Theorem~\ref{male}, there exists $J\in GL(X)$ such that $J(A)=B$.
Thus $GL(X)$ acts transitively on $\Sigma(X)$. Since $X$ is
separable and metrizable, $\Sigma(X)$ is non-empty. Hence $X$ is a
G-space, which proves the implication
(\ref{mama}.1)$\Longrightarrow$(\ref{mama}.2). Since every separable
infinite dimensional Fr\'echet space supports a hypercyclic operator
\cite{bope}, Lemma~\ref{nle} provides the implication
$(\ref{mama}.2)\Longrightarrow(\ref{mama}.3)$.

\subsection{Proof of the implication
(\ref{mama}.3)$\Longrightarrow$(\ref{mama}.1)}

Let $X$ be a separable Fr\'echet space possessing no continuous
norm. The implication (\ref{mama}.3)$\Longrightarrow$(\ref{mama}.1)
will be verified if we show that there exists $A\in\Sigma(X)$, which
is not an orbit of a continuous linear operator. If $X$ is
isomorphic to $\omega$, the job is already done by Bonet, Frerick,
Peris and Wengenroth \cite[Proposition~3.3]{fre}. It remains to
consider the case of $X$ non-isomorphic to $\omega$. Since $X$ is a
Fr\'echet space possessing no continuous norm and non-isomorphic to
$\omega$, the topology of $X$ can be defined by an increasing
sequence $\{p_n\}_{n\in\N}$ of seminorms such that $p_1$ is
non-trivial and $\ker p_n/\ker p_{n+1}\neq \{0\}$ for each $n\in\N$.
By Lemma~\ref{lile}, there is a dense in $X$ countable
$p_1$-independent set $B$. Since $\ker p_n/\ker p_{n+1}\neq \{0\}$
for each $n\in\N$, for each $n\in\N$, we can pick $x_n\in\ker
p_n\setminus\ker p_{n+1}$. Let $C=\{x_n:n\in\N\}$ and $A=B\cup C$.
Obviously $A$ is a countable subset of $X$. Since $B$ is dense in
$X$ and $B\subseteq A$, $A$ is dense in $X$. Finally, the
$p_1$-independence of $B$ and the inclusions $x_n\in\ker
p_n\setminus\ker p_{n+1}$ imply that $A$ is linearly independent.
Thus $A\in \Sigma(X)$. It suffices to verify that $A$ is not an
orbit. Assume the contrary. Then there are $T\in L(X)$ and $x\in X$
such that $A=O(T,x)$. Let $M=\{n\in\Z_+:T^nx\in C,\ T^{n+1}x\in
B\}$. Since $B$ does not meet $\ker p_1$, $p_1(T^{n+1}x)>0$ for
every $n\in M$. Thus we can consider the (finite or countable)
series $S=\sum\limits_{n\in M}\frac{T^nx}{p_1(T^{n+1}x)}$. Since
$T^nx$ for $n\in M$ are pairwise distinct elements of $C$ and every
$p_k$ vanishes on all but finitely many elements of $C$, the series
$S$ converges absolutely in $X$. Since $T:X\to X$ is a continuous
linear operator and every continuous linear operator on a locally
convex space maps an absolutely convergent series to an absolutely
convergent series, the series $T(S)=\sum\limits_{n\in
M}\frac{T^{n+1}x}{p_1(T^{n+1}x)}$ is also absolutely convergent.
Hence the application of $p_1$ to the terms of $T(S)$ gives a
convergent series of non-negative numbers. But the latter series is
$\sum\limits_{n\in
M}\frac{p_1(T^{n+1}x)}{p_1(T^{n+1}x)}=\sum\limits_{n\in M} 1$. Its
convergence is equivalent to the finiteness of $M$. Thus $M$ is
finite. Let $m=\max(M)$ if $M\neq\varnothing$ and $m=0$ if
$M=\varnothing$. Since $C\subset O(T,x)$ and $C$ is infinite, there
is $k\in\Z_+$ such that $k>m$ and $T^kx\in C$. Since
$M\cap\{j\in\Z_+:j\geq k\}=\varnothing$, from the definition of $M$
it follows that $T^jx\in C$ for every $j\geq k$. Hence $T^jx\in C$
for all but finitely many $j$. It follows that $B=O(T,x)\setminus C$
is finite, which is a contradiction. This contradiction shows that
$A$ is not an orbit and completes the proof of the implication
(\ref{mama}.3)$\Longrightarrow$(\ref{mama}.1) and that of
Theorem~\ref{mama}.

\section{Proof of Theorem~\ref{mamama} \label{s5}}

\begin{lemma}\label{bububu} Let $p$ be a continuous seminorm on a
locally convex space $X$ and $E$ be a countably dimensional subspace
of $X$ such that $E\cap \ker p=\{0\}$. Then there exist a Hamel
basis $\{u_n\}_{n\in\N}$ in $E$ and a sequence $\{f_n\}_{n\in\N}$ in
$X'_p$ such that $f_n(u_m)=\delta_{n,m}$ for every $m,n\in\N$.
\end{lemma}

\begin{proof} Begin with an arbitrary Hamel basis $\{y_n\}_{n\in\N}$
in $E$. The proof is a variation of the Gramm--Schmidt procedure.
Clearly, it suffices to construct (inductively) two sequences
$\{u_n\}_{n\in\N}$ in $E$ and $\{f_n\}_{n\in\N}$ in $X'_p$ such that
for every $n\in\N$,
\begin{align}
&u_n\in y_n+\spann\{y_j:j<n\};\label{III1}
\\
&f_j(u_k)=\delta_{j,k}\ \ \text{for $j,k\leq n$}.\label{III2}
\end{align}
Indeed, (\ref{III1}) ensures that $\{u_n:n\in\N\}$ is also a Hamel
basis in $E$.

First, we set $u_1=y_1$ and note that $p(u_1)\neq 0$. Then we use
the Hahn--Banach theorem to find $f_1\in X'_p$ such that
$f_1(u_1)=1$. This gives us the basis of induction. Assume now that
$m\geq 2$ and $u_n,f_n$ satisfying (\ref{III1}) and (\ref{III2}) for
$n<m$ are already constructed. Condition (\ref{III2}) for $n<m$
allows us to pick $u_m\in y_m+\spann\{y_n:n<m\}$ such that
$f_j(u_n)=0$ for every $j<n$. Since $y_n$ are linearly independent,
$u_m\in E\setminus\{0\}$. Since $E\cap \ker p=\{0\}$, $p(u_m)\neq
0$. Since $u_1,\dots,u_m$ are linearly independent elements of $E$
and $p(u_m)\neq 0$, the Hahn--Banach theorem allows us to choose
$f_m\in X'_p$ such that $f_m(u_m)=1$ and $f_m(u_j)=0$ for $j<m$.
Clearly, $u_n$ and $f_n$ for $n\leq m$ satisfy (\ref{III1}) and
(\ref{III2}) for $n\leq m$. This completes the inductive procedure
of constructing the sequences $\{u_n\}_{n\in\N}$ in $E$ and
$\{f_n\}_{n\in\N}$ in $X'_p$ satisfying (\ref{III1}) and
(\ref{III2}) for every $n\in\N$.
\end{proof}

The following lemma features as \cite[Theorem~2.2]{bama-book}.

\begin{lemma}\label{bml} Let $X$ be a separable Fr\'echet space and $T\in
L(X)$ be such that the linear span of the union of $T^n(X)\cap \ker
T^n$ for $n\in\N$ is dense in $X$. Then $I+T$ is hypercyclic.
\end{lemma}

\begin{lemma}\label{hype} Let $p$ be a non-trivial continuous
seminorm on a separable locally convex space $X$ for which there
exists a Banach disk $D$ in $X$ such that $X_D$ is a dense subspace
of $X$ and the Banach space $(X_D,p_D)$ is separable. Then there
exists $T\in L(X)$ such that $T$ is hypercyclic and $Tx=x$ for every
$x\in \ker p$.
\end{lemma}

\begin{proof} Since $X_D$ is dense in $X$ and the Banach space
topology on $X_D$ is stronger than the one inherited from $X$, the
restriction of $p$ to $X_D$ is a non-trivial continuous seminorm on
the Banach space $X_D$. By Lemma~\ref{l22}, there is a dense
countable subspace $A$ of the Banach space $X_D$ such that $A$ is
$p$-independent. Let $E=\spann(A)$. Then $E$ is a dense in
$(X_D,p_D)$ and therefore in $X$ countably dimensional subspace of
$X_D$. Since $A$ is $p$-independent, $E\cap\ker p=\{0\}$. By
Lemma~\ref{bububu}, there is a Hamel basis $\{u_n\}_{n\in\N}$ in $E$
and a sequence $\{f_n\}_{n\in\N}$ in $X'_p$ such that
$f_n(u_m)=\delta_{n,m}$ for every $m,n\in\N$.

Consider the linear map $S:X\to X_D$ defined by the formula:
$$
Sx=\sum_{n=1}^\infty
\frac{2^{-n}f_{n+1}(x)}{p_D(u_n)p^*(f_{n+1})}u_n.
$$
The series in the above display converges absolutely in $X_D$ since
$|f_{n+1}(x)|\leq p(x)p^*(f_{n+1})$. Furthermore $p_D(Sx)\leq p(x)$
for every $x\in X$. Hence $S$ is a well-defined continuous linear
map from $X$ to $X_D$. In particular, $S\in L(X)$ and the
restriction $S_D$ of $S$ to $X_D$ is a continuous linear operator on
the Banach space $X_D$. Moreover, analyzing the action of $S$ on
$u_k$, it is easy to see that $S(E)=S_D(E)=E$ and therefore
$E\subseteq S_D^n(X_D)$ for every $n\in\N$. Furthermore, $u_n\in\ker
S_D^n$ for every $n\in\N$. Hence $E\subseteq
\bigcup\limits_{n\in\N}\ker S_D^n$. Since $E$ is dense in $X_D$,
Lemma~\ref{bml} implies that $T_D=I+S_D$ is a hypercyclic operator
on the Banach space $X_D$. Since the topology of $X_D$ is stronger
than the one inherited from $X$ and $X_D$ is dense in $X$, every
hypercyclic vector for $T_D$ is also hypercyclic for $T=I+S\in
L(X)$. Thus $T=I+S$ is hypercyclic. Next, $p$-boundedness of each
$f_k$ implies that each $f_k$ vanishes on $\ker p$. Hence $\ker
p\subseteq \ker S$ and therefore $Tx=x$ for every $x\in\ker p$.
\end{proof}

Now we are ready to prove Theorem~\ref{mamama}. Let $E$ be a
countably dimensional metrizable locally convex space. Denote the
completion of $E$ by the symbol $X$. That is, $X$ is a separable
infinite dimensional Fr\'echet space and $E$ is a dense countably
dimensional subspace of $E$.

{\bf Case 1:} \ $X$ is non-isomorphic to $\omega$. In this case the
topology of $X$ is non-weak and therefore $X$ supports a non-trivial
continuous seminorm $p$. By Lemma~\ref{lile}, there is a countable
dense in $X$ $p$-independent set $B$ such that $B\subseteq E$. A
standard application of Zorn's lemma provides a maximal by inclusion
$p$-independent subset $A$ of $E$ containing $B$. Since $B\subseteq
A$, $A$ is dense in $X$. Since $E$ is countably dimensional, $A$ is
countable ($p$-independence implies linear independence). By
Lemma~\ref{l22}, every separable infinite dimensional Fr\'echet
space contains a Banach disk $K$ such that $X_K$ is a separable
Banach space and $X_K$ is dense in $X$. Now by Lemma~\ref{hype}
there is a hypercyclic $T\in L(X)$ such that $Tx=x$ for every $x\in
\ker p$. Let $u$ be a hypercyclic vector for $T$. First, we shall
verify that $O(T,u)$ is $p$-independent. Assume the contrary. Then
there exists a non-zero polynomial $r$ such that $r(T)u\in\ker p$.
Then for every $n\in\Z_+$, we can write $t^n=r(t)q(t)+v(t)$, where
$q$ and $v$ are polynomials and $\deg v<\deg r=d$. Hence
$T^nu=q(T)r(T)u+v(T)u$. Since $r(T)u\in\ker p$ and $\ker p$ is
invariant for $T$, $q(T)r(T)u\in\ker p$. Hence $O(T,u)\subseteq
L+\ker p$, where $L=\spann\{u,Tu,\dots,T^{d-1}u\}$. Since $L$ is
finite dimensional and $\ker p$ is a closed subspace of $X$ of
infinite codimension, $L+\ker p$ is a proper closed subspace of $X$.
We have obtained a contradiction with the density of $O(T,u)$. Thus
the countable dense in $X$ set $O(T,u)$ is $p$-independent. Recall
that $A$ is also countable, dense in $X$ and $p$-independent. By
Lemma~\ref{l22}, there is a Banach disk $D$ in $X$ such that both
$A$ and $O(T,u)$ are dense subsets of the Banach space $(X_D,p_D)$.
By Theorem~\ref{male}, there exists $J\in GL(X)$ such that
$J(O(T,u))=A$ and $Jx=x$ for every $x\in\ker p$. Let $S=JTJ^{-1}$.
Exactly as in the proof of Lemma~\ref{nle}, one easily sees that
$Ju$ is a hypercyclic vector for $S$ and that $O(S,Ju)=A$. In
particular, $Ju\in A\subset E$. It remains to verify that
$S(E)\subseteq E$. Indeed, in this case the restriction of $S$ to
$E$ provides a continuous linear operator on $E$ with $Ju$ being its
hypercyclic vector.

Let $x\in E$. It suffices to show that $Sx\in E$. The maximality of
$A$ implies that we can write $x=y+z$, where $y\in \spann(A)$ and
$z\in\ker p$. Since $A\subset E$, $y\in E$ and therefore $z=x-y\in
E$. Since $A=O(S,Ju)$, $S(A)\subseteq A$. Hence
$S(\spann(A))\subseteq \spann(A)\subseteq E$. It follows that $Sy\in
E$. Since $Tv=Jv=v$ for $v\in\ker p$, we have $Sv=v$ for $v\in \ker
p$ and therefore $Sz=z$. Thus $Sx=Sy+Sz=Sy+z\in E$, as required.
This completes the proof for Case~1.

{\bf Case 2:} \ $X$ is isomorphic to $\omega$. It is well-known
(see, for instance, \cite{bope}) that $\omega$ supports a
hypercyclic operator. Actually, it is easy to see that the shift
$S\in L(\K^\N)$, $(Sx)_n=x_{n+1}$ is hypercyclic. Thus, we can take
$S\in L(X)$ and $x\in X$ such that $x$ is a hypercyclic vector for
$S$ and let $F=\spann(O(S,x))$. Then $F$ is another dense countably
dimensional subspace of $X$. Obviously $F$ supports a hypercyclic
operator (the restriction of $S$ to $F$). By Theorem~\ref{omeg}, $E$
and $F$ are isomorphic. Hence $E$ supports a hypercyclic operator.
The proof of Theorem~\ref{mamama} is now complete.

\section{Open problems and remarks \label{s6}}

Note that the locally convex direct sum $\phi$ of countably many
copies of the one-dimensional space $\K$ is a complete countably
dimensional locally convex space. A number of authors, see, for
instance, \cite{bope}, have observed that $\phi$ supports no
hypercyclic operators.

\begin{question}\label{q1} Characterize countably dimensional
locally convex spaces supporting a hypercyclic operator.
\end{question}

The following is an interesting special case of the above problem.

\begin{question}\label{q2} Are there any complete countably dimensional
locally convex spaces supporting a hypercyclic operator?
\end{question}

The following question also seems to be interesting.

\begin{question}\label{q3} Characterize complete G-spaces.
Characterize complete G-spaces supporting a hypercyclic operator.
\end{question}

Note that although $\omega$ is not a G-space, Theorem~\ref{omeg}
shows that $GL(\omega)$ acts transitively on the set of dense
countably dimensional subspaces of $\omega$.

\begin{question}\label{q4} Characterize complete locally convex
spaces $X$ with the property that $GL(X)$ acts transitively on the
set of dense countably dimensional subspaces of $X$.
\end{question}

%

\small\rm

\vskip1truecm

\scshape

\noindent Andre Schenke \ {\rm and} \ Stanislav Shkarin

\noindent Queen's University Belfast

\noindent Pure Mathematics Research Centre

\noindent University road, Belfast, BT7 1NN, UK

\noindent E-mail address: \qquad {\tt s.shkarin@qub.ac.uk,\qquad
aschenke01@qub.ac.uk}

\end{document}